\documentclass{article}

\usepackage{amsmath,amssymb}
\usepackage{amsthm}
\usepackage[singlelinecheck=off]{caption}
\usepackage{microtype}
\usepackage{graphicx}
\usepackage{algorithm,algpseudocode}
\usepackage{mathrsfs}
\usepackage{epstopdf}
\usepackage{subcaption}
\usepackage{float}
\usepackage{url}
\usepackage{xcolor}
\usepackage{enumitem}

\newcommand{\mc}[1]{\mathcal{#1}}

\newtheorem{theorem}{Theorem}
\newtheorem{lemma}{Lemma}
\newtheorem{assumption}{Assumption}

\title{Solving PDE problems with uncertainty using neural-networks}

\author{ Yuehaw Khoo\thanks{Department of Mathematics, Stanford
    University, Stanford, CA 94305, USA (\texttt{ykhoo@stanford.edu}).}
  \and Jianfeng Lu \thanks{Department of Mathematics, Department of
    Chemistry and Department of Physics, Duke University, Durham, NC
    27708, USA (\texttt{jianfeng@math.duke.edu}).}  \and Lexing
  Ying\thanks{Department of Mathematics and ICME, Stanford University,
    Stanford, CA 94305, USA (\texttt{lexing@stanford.edu}).}  }

\begin{document}

\maketitle

\begin{abstract}
	The curse of dimensionality is commonly encountered in numerical partial differential equations
	(PDE), especially when uncertainties have to be modeled into the equations as random
	coefficients. However, very often the variability of physical quantities derived from a PDE can be
	captured by a few features on the space of the coefficient fields. Based on such an observation,
	we propose using a neural-network (NN) based method to parameterize the physical quantity of
	interest as a function of input coefficients. The representability of such quantity using a
	neural-network can be justified by viewing the neural-network as performing time evolution to find
	the solutions to the PDE.  We further demonstrate the simplicity and accuracy of the approach
	through notable examples of PDEs in engineering and physics.
\end{abstract}

\section{Introduction}
Uncertainty quantifications in physical and engineering applications often involve the study of
partial differential equations (PDE) with random coefficient field. To understand the behavior of a
system in the presence of uncertainties, one can extract PDE-derived physical quantities as
functionals of the coefficient fields. This can potentially require solving the PDE an exponential
number of times numerically even with a suitable discretization of the PDE domain, and of the range
of random variables. Fortunately in most PDE applications, often these functionals depend only on a
few characteristic ``features'' of the coefficient fields, allowing them to be determined from
solving the PDE a limited number of times.

A commonly used approach for uncertainty quantifications is Monte-Carlo sampling. An ensemble of
solutions is built by repeatedly solving the PDE with different realizations of the coefficient
field. Then physical quantities of interest, for example, the mean of the solution at a given
location, can be computed from the ensemble of solutions. Although being applicable in many
situations, the computed quantity is inherently noisy. Moreover, this method lacks the ability to
obtain new solutions if they are not sampled previously. Other approaches exploit the low underlying
dimensionality assumption in a more direct manner. For example the stochastic Galerkin method
\cite{matthies2005galerkin,stefanou2009stochastic} expands the random solution using certain
prefixed basis functions (i.e. polynomial chaos \cite{wiener1938homogeneous,xiu2002wiener}) on the
space of random variables, thereby reducing the high dimensional problem to a few deterministic
PDEs. Such type of methods requires careful treatment of the uncertainty distributions, and since
the basis used is problem independent, the method could be expensive when the dimensionality of the
random variables is high. There are data-driven approaches for basis learning such as applying
Karhunen-Lo\`eve expansion to PDE solutions from different realizations of the PDE
\cite{cheng2013data}. Similarly to the related principal component analysis, such linear
dimension-reduction techniques may not fully exploit the nonlinear interplay between the random
variables. At the end of day, the problem of uncertainty quantification is one of characterizing the
low-dimensional structure of the coefficient field that gives the observed quantities.


On the other hand, the problem of dimensionality reduction has been central to the fields of
statistics and machine learning. The fundamental task of regression seeks to find a function
$h_\theta$ parameterized by a parameter vector $\theta\in\mathbb{R}^p$ such that
\begin{equation}
\label{ls error}
f(a) \approx h_\theta(a),\ a\in\mathbb{R}^q.
\end{equation}
However, choosing a sufficiently large class of approximation functions without the issue of
over-fitting remains a delicate business, for example when choosing the set of basis $\{\phi_k(a)\}$
such that $f(a) = \sum_{k} \beta_k \phi_k(a)$ in linear regression. In the last decade, deep
neural-networks have demonstrated immense success in solving a variety of difficult regression
problems related to pattern recognitions
\cite{hinton2006reducing,lecun2015deep,schmidhuber2015deep}. A key advantage of using neural-network
is that it bypasses the traditional need to handcraft basis for spanning $f(a)$ as in linear
regression but instead, directly learns an approximation that satisfies (\ref{ls error}) in a
data-driven way. The performance of neural-network in machine learning applications, and more
recently in physical applications such as representing quantum many-body states
(e.g. \cite{carleo2017solving,torlai2016learning}), encourages us to study its use in the context of
solving PDE with random coefficients. More precisely, we want to learn $f(a)$ that maps the random
coefficient vector $a$ in a PDE to some physical quantities described by the PDE.

Our approach to solving quantities arise from PDE with randomness is
conceptually simple, consisting of the following steps:
\begin{itemize}
	\item Sample the random coefficients ($a$ in (\ref{ls error})) of the PDE from a user-specified
	distribution. For each set of coefficients, solve the deterministic PDE to obtain the physical
	quantity of interest ($f(a)$ in (\ref{ls error})).
	\item Use a neural-network as the surrogate model $h_{\theta}(a)$ in (\ref{ls error}) and train it
	using the previously obtained samples.
	\item Validate the surrogate forward model with more samples. The neural network is now ready for
	applications.
\end{itemize}
Though being a simple method, to the best of our knowledge, dimension reduction based on
neural-network representation has not been adapted to solving PDE with uncertainties. We demonstrate
the success of neural-network in two important PDEs that have wide applications in physics and
engineering. In particular, we consider solving for the effective conductance in inhomogeneous media
and the ground state energy of a nonlinear Schr\"odinger equation (NLSE) having inhomogeneous
potential. These quantities are $f$ in \eqref{ls error} that we want to learn as a function of $a$,
where $a$ is the random conductivity coefficient or the random potential. The main contributions of
our work are
\begin{itemize}
	\item We provide theoretical guarantees on neural-network representation of $f(a)$ through explicit
	construction for the parametric PDE problems under study;
	\item We show that even a rather simple neural-network architecture can learn a good representation
	of $f(a)$ through training.
\end{itemize}
We note that our work is different from \cite{E2018,khoo2018solving,
	lagaris1998artificial,long2017pde,rudd2015constrained}, which solve
deterministic PDE numerically using a neural-network. The goal of
these works is to parameterize the solution of a deterministic PDE
using neural-network and use optimization methods to solve for the PDE
solution. It is also different from \cite{han2017deep} where a
deterministic PDE is solved as a stochastic control problem using
neural-network. In this paper, the function that we want to
parameterize is over the coefficient field of the PDE.

The advantages of having an explicitly parameterized approximation to $f(\cdot)$ are numerous, which
we will only list a couple here. First, the neural-network parameterized function can serve as a
surrogate forward model for generating samples cheaply for statistical analysis. Second, the task of
optimizing some function of the physical quantity with respect to the PDE coefficients in
engineering design problems can be done with the help of a gradient calculated from the
neural-network. To summarize, obtaining a neural-network parametrization could limit the use of
expensive PDE solvers in applications.

The paper is organized as followed. In Section \ref{section: two examples}, we provide background on
the two PDEs of interest. In Section \ref{section: theory}, the theoretical justification of using
NN to represent the physical quantities derived from the PDEs introduced in Section \ref{section:
	two examples}, is provided. In Section \ref{section: framework}, we describe the neural-network
architecture for handling these PDE problems and report the numerical results. We finally conclude
in Section \ref{section: conclusion}.

\section{Two examples of parametric PDE problems}
\label{section: two examples}
This section introduces the two PDE models -- the linear elliptic equation \cite{larsson2008partial}
and the nonlinear Schr\"odinger equation \cite{kato1989nonlinear} -- we want to solve for. Elliptic
equations are commonly used to study steady heat conduction in a given material. When the material
has inhomogeneities (modeled as random conductivity coefficients in the elliptic equation), one
typically wants to understand the effective conductivity of the material. NLSE is used to understand
light propagation in waveguide and also the quantum mechanical phenomena where bosonic particles
highly concentrate in the lowest-energy state (Bose-Einstein condensation). We want to study how the
energy of such ground state behaves when the NLSE is subjected to random potential field. Therefore,
we focus on the map from the coefficient field of these PDEs to their relevant physical
quantities. In both of the PDEs, the boundary condition is taken to be periodic for simplicity.

\subsection{Effective coefficients for inhomogeneous elliptic equation}\label{section: random elliptic}
Our first example will be finding the effective conductance/coefficient in a non-homogeneous
media. For this, we consider a class of coefficient functions
\begin{equation}
\label{coef assumption}
\mathcal{A} = \{ a \in L^{\infty}([0, 1]^d) \mid \lambda_1 \geq a(x) \geq \lambda_0 > 0 \},
\end{equation}
for some fixed constants $\lambda_0$ and $\lambda_1$. Fix a direction $\xi \in \mathbb{R}^d$ with
$\|\xi\|_2=1$ ($\| \cdot \|_2$ is the Euclidean norm). We want to obtain the effective conductance
functional $A_{\mathrm{eff}}:\mathcal{A}\rightarrow \mathbb{R}$ defined by
\begin{equation}
\label{effective conductance}
A_{\mathrm{eff}}(a) = \min_{u(x)} \int_{[0,1]^d} a(x)\| \nabla u(x)+\xi \|_2^2 \, \mathrm{d} x.
\end{equation}
The minimizer $u_a(x)$ of this variational problem (here the subscript ``$a$'' in $u_a$ is used to
denote its dependence on the coefficient field $a$) satisfies the following elliptic partial
different equation
\begin{equation}
\label{random elliptic}
- \nabla \cdot \left( a(x) (\nabla u(x)+\xi)\right) = 0,\quad x\in[0,1]^d
\end{equation}
with periodic boundary condition. With $u_a$ available, one obtains
\[
A_{\mathrm{eff}}(a) = \int_{[0,1]^d} a(x)\| \nabla u_a(x)+\xi\|_2^2 \, \mathrm{d} x.
\]
In practice, to parameterize ${A}_{\mathrm{eff}}$ as a functional of the coefficient field
$a(\cdot)$, we discretize the domain using a uniform grid with step size $h=1/n$ and grid points
denoted by $x_i = ih$, where the multi-index
$i\in \{(i_1,\ldots,i_d)\}, 1\le i_1,\ldots,i_d\le n$. In this way, we can think about the
coefficient field $a(x)$ and the solution $u(x)$ represented on the grid points both as vectors with
length $n^d$. More precisely, we redefine
\[
\mathcal{A} = \{a\in\mathbb{R}^{n^d} \mid a_i \in [\lambda_0, \lambda_1], \, \forall i\}
\]
and discretize the term $-\nabla \cdot \left(a(x) \nabla u(x)\right)$ using central difference
\begin{equation}
- \sum_{k=1}^d \frac{a_{i+e_k/2}(u_{i+e_k}-u_{i}) - a_{i-e_k/2} (u_i-u_{i-e_k})}{h^2},
\end{equation}
for each $i$, where $\{e_k\}_{k=1}^d$ denotes the canonical basis in $\mathbb{R}^d$ and each value
of $a$ at a half grid point is obtained by averaging the values at its two nearest grid points. Then
the discrete version of \eqref{random elliptic} is the linear system $L_a u = b_a$ with
\begin{align}\label{discrete elliptic}
&(L_a u)_i:=\sum_{k=1}^d \frac{-a_{i+e_k/2}u_{i+e_k}  + (a_{i-e_k/2}+a_{i+e_k/2})u_i - a_{i-e_k/2}u_{i-e_k}}{h^2}\\
&(b_a)_i  :=\sum_{k=1}^d  \frac{\xi_k (a_{i+e_k/2} - a_{i-e_k/2} )}{h}.
\end{align}
From \eqref{effective conductance}, one can see that the discrete version of the effective
conductivity, also denoted by $A_{\mathrm{eff}}(a)$ for $a\in\mathbb{R}^{n^d}$, can be obtained from
solving the discrete variational problem
\begin{equation}\label{effective conductance strong} 
A_{\mathrm{eff}}(a) = 2 \min_{u\in\mathbb{R}^{n^d}} \mc{E}(u; a),\quad
\mc{E}(u; a):=\frac{h^d}{2} (u^{\top} L_a u - 2 u^{\top} b_a +  a^\top \mathbf{1}),
\end{equation}
or equivalently, solving $u_a$ from $L_au = b_a$ and setting
\[
A_{\mathrm{eff}}(a)=h^d (u_a^{\top} L_a u_a - 2 u_a^{\top} b_a +  a^\top \mathbf{1}).
\]

We stress that in order to simplify the notations, we use $u$ and $a$ as vectors in
$\mathbb{R}^{n^d}$, although they were previously used as functions in \eqref{effective
	conductance}. The interpretation of $u$ and $a$ as functions or vectors should be clear from the
context.

\subsection{NLSE with inhomogeneous background potential}\label{section: random NLSE}
For the second PDE example, we want to find the ground state energy $E_0$ of a nonlinear Schr\"odinger
equation with potential $a(x)$:
\begin{equation} 
\label{NLSE}
-\Delta u(x) + a(x) u(x) + \sigma u(x)^3 = E_0 u(x),\quad x\in[0,1]^d, \\
\text{s.t.}\ \int_{[0,1]^d} u(x)^2 \, \mathrm{d} x= 1.
\end{equation}
We take $\sigma = 2$ in this work and thus consider a defocusing cubic Schr\"odinger equation, which
can be understood as a model for soliton in nonlinear photonics or Bose-Einstein condensate with
inhomogeneous media. Similar to \eqref{discrete elliptic}, we solve the discretized version of the
NLSE
\begin{equation}
\label{discrete NLSE}
(L u)_i + a_i u_i + \sigma u_i^3 = E_0 u_i,\quad
\sum_{i=1}^{n^d} u_i^2 h^d= 1,\quad  (Lu)_i :=\sum_{k=1}^d \frac{-u_{i+e_k} + 2u_i -u_{i-e_k}}{h^2}.
\end{equation}

Due to the nonlinear cubic term, it is more difficult to solve for the NLSE numerically compare to
(\ref{random elliptic}). Therefore in this case, the value of having a surrogate model of $E_0$ as a
function of $a$ is more significant. We note that the solution $u$ to \eqref{discrete NLSE} (and
thus $E_0$) can also be obtained from the following variational problem
\begin{equation}
\min_{u\in\mathbb{R}^{n^d}:\, \|u\|_2^2 = n^d} u^\top L u + u^\top \text{diag}(a) u +  \frac{\sigma}{2} \sum_i u_i^4,
\end{equation}
where the $\text{diag}(\cdot)$ operator forms a diagonal matrix given a vector.

\section{Theoretical justification of deep neural-network representation}
\label{section: theory}

The physical quantities introduced in Section \ref{section: two examples} are determined through the
solution of the PDEs given the coefficient field. Rather than solving the PDE, we will prove that
the map from coefficient field to such quantities can be represented using convolutional NNs. The
main idea is to view the solution $u$ of the PDE as being obtained via time evolution, where each
layer of the NN corresponds to the solution at discrete time step. In other words, mapping the input
$a$ from the first layer to last layer in the NN resembles the time-evolution of a PDE with discrete
time-steps. We focus here on the case of solving elliptic equations with inhomogeneous coefficients.
Similar line of reasoning can be used to demonstrate the representability of the ground state-energy
$E_0$ as a function of $a$ using an NN.

\begin{theorem}
	\label{theorem: effective conductance}
	Fix an error tolerance $\epsilon > 0$, there exists a neural-network $h_\theta(\cdot)$ with
	$O(n^d)$ hidden nodes per-layer and
	$O((\frac{\lambda_1}{\lambda_0}+1)\frac{n^{2}}{\epsilon})$ layers such that for
	any $a\in\mathcal{A}=\{a\in\mathbb{R}^{n^d} \mid a_i \in [\lambda_0, \lambda_1],\,\forall i\}$, we
	have
	\begin{equation}
	\lvert h_\theta(a)  - A_{\mathrm{eff}}(a) \rvert \leq \epsilon \lambda_1.
	\end{equation} 
\end{theorem}

Note that due to the ellipticity assumption $a \in \mathcal{A}$, the effective conductivity is
bounded from below by $A_{\mathrm{eff}}(a) \geq \lambda_0 > 0$. Therefore the theorem immediately
implies a relative error bound
\begin{equation}
\frac{\lvert h_\theta(a)  - A_{\mathrm{eff}}(a) \rvert}{A_{\mathrm{eff}}(a)} \leq \epsilon \frac{\lambda_1}{\lambda_0}.
\end{equation}

We illustrate the main idea of the proof in the rest of the section,
the technical details of the proof are deferred to the supplementary
materials.

The first observation is that, due to the variational characterization \eqref{effective conductance
	strong}, in order to get $A_{\mathrm{eff}}(a)$ we may minimize $\mc{E}(u; a)$ over the solution
space $u$, using e.g., steepest descent:
\begin{equation}\label{discrete time heat equation}
\begin{aligned}
u^{m+1} & = u^{m} - \Delta t \frac{\partial \mc{E}(u^m; a)}{\partial u}\\
& = u^{m} - \Delta t \bigl(L_a u^m - b_{a} \bigr),
\end{aligned}
\end{equation}
where $\Delta t$ is a step size chosen sufficiently small to ensure descent of the energy. Note that
the optimization problem is convex due to the ellipticity assumption \eqref{coef assumption} of the
coefficient field $a$ (which ensures $u^\top L_a u>0$ except for $u=\mathbf{1}$) with Lipshitz
continuous gradient, therefore the iterative scheme converges to the minimizer with proper choice of
step size for any initial condition. Thus we can choose $u^0 = 0$.


Now we identify the iteration scheme in \eqref{discrete time heat
	equation} with a convolutional NN architecture (Fig. \ref{figure:
	time_evol}) by viewing $m$ as an index of the NN layers. The input
of the NN is the vector $a\in \mathbb{R}^{n^d}$, and the hidden layers
are used to map between the consecutive pairs of $(d+1)$-tensors
$U^m_{i_0i_1\ldots i_{d}}$ and $U^{m+1}_{i_0i_1\ldots i_{d}}$.  The
zeroth dimension for each tensor $U^m$ is the channel dimension and
the last $d$ dimensions are the spatial dimensions. If we let the
channels in each $U^m$ be consisted of a copy of $a$ and a copy of
$u^m$, e.g., let
\begin{equation}
U^m_{0 i_1\ldots i_{d}} = a_{(i_1,\ldots, i_{d})},\quad U^m_{1 i_1\ldots i_{d}} = u^m_{(i_1,\ldots, i_{d})},
\end{equation}
in light of \eqref{discrete time heat equation} and \eqref{discrete elliptic}, one simply needs to
perform local convolution (to aggregate $a$ locally) and nonlinearity (to approximate quadratic form
of $a$ and $u^m$) to get from $U^m_{1i_1\ldots i_{d}} = u^m_{(i_1,\ldots, i_{d})} $ to
$U^{m+1}_{1i_1\ldots i_{d}} = u^{m+1}_{(i_1,\ldots, i_{d})} $; while the $0$-channel is simply
copied to carry along the information of $a$. Stopping at $m=M$ layer and letting $u^M$ be the
approximate minimizer of $\mc{E}(u;a)$, based on \eqref{effective conductance strong}, we let
$\mc{E}(u^M;a)$ be an approximation to $A_{\mathrm{eff}}(a)$. This architecture of NN to approximate
the effective conductance is illustrated in Fig. \ref{figure: time_evol}. Note that the architecture
of NN used in the proof resembles a deep ResNet \cite{he2016deep}, as the coefficient field $a$ is
passed from the first to the last layer. The detailed estimates of the approximation error and the
number of parameters will be deferred to the supplementary materials.

Let us point out that if we take the continuum time limit of the steepest descent dynamics, we
obtain a system of ODE
\begin{equation}
\partial_t u = -(L_a u-b_{a}), 
\end{equation}
which can be viewed as a spatially discretized PDE. Thus our construction of the neural network in
the proof is also related to the work \cite{long2017pde} where multiple layers of convolutional NN
is used to learn and solve evolutionary PDEs. However, the goal of the neural network here is to
approximate the physical quantity of interest as a functional of the (high-dimensional) coefficient
field, which is quite different from the view point of \cite{long2017pde}.

\begin{figure}
	\centering
	\includegraphics[width=0.7\textwidth]{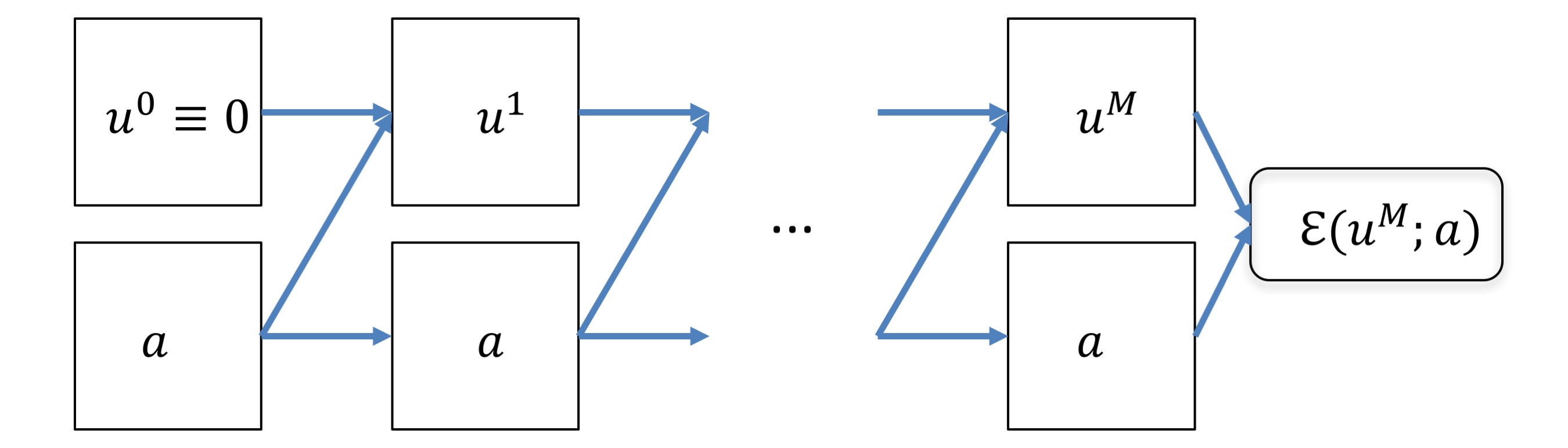}
	\caption{Construction of the NN in the proof of Theorem
		\ref{theorem: effective conductance}. The NN takes the coefficient
		field $a$ as an input and the convolutional and local nonlinearity
		layers are used to map from $u^m, a$ to $u^{m+1},a$, $m=0,M-1$. At
		$u_M$, local convolutions and nonlinearity are used to obtained
		$\mc{E}(u^M;a)$.}\label{figure: time_evol}
\end{figure}

We also remark that the number of layers of the NN required by
Theorem~\ref{theorem: effective conductance} is rather large. This is
due to the choice of the (unconditioned) steepest descent algorithm as
the engine of optimization to generate the neural network architecture
used in the proof. Nevertheless, the number of parameters required in
the NN representation is still much fewer than the exponential scaling
in terms of the input dimension of a shallow network \cite{Barron1993}
from universal approximation theorem. Moreover, with a better
preconditioner such as the algebraic multigrid \cite{xu2017algebraic}
for the gradient, we can effectively reduce the number of layers to
$O(1)$ and thus achieves an optimal count of parameters involved in
the NN; the details will be left for future works. In practice, as
shown in the next section by actual training of parametric PDEs, the
neural network architecture can be much simplified while maintaining a
good approximation to the quantity of interest.

\section{Proposed network architecture and numerical results}
\label{section: framework}
In this section, based on the discussion in Section \ref{section: theory}, we propose using
convolutional NN to approximate the physical quantities given by the PDE with a periodic boundary
condition. We first describe the architecture of the neural network in Section \ref{section: archi},
then the implementation details and numerical results are provided in Section \ref{section:
	implementation} and \ref{section: numerics} respectively.

\subsection{Architecture} \label{section: archi}

In Fig. \ref{figure: onelayer}, we show the architecture for the 2D case with domain being a unit
square, though it can be generalized to solving PDEs in any dimensions. The input to the NN is an $n\times n$
matrix representing the coefficient field $a\in \mathbb{R}^{n^2}$ on grid points, and the
output of the network gives physical quantity of interest from the PDE. The main part of the network
are convolutional layers with ReLU being the nonlinearity. This extracts the relevant features of
the coefficient field around each grid point that contribute to the final output. The use of
a sum-pooling followed by a linear map to obtain the final output is based on the translational
symmetry of the function $f(\cdot)$ to be represented. More precisely, let
$a^{\tau_1\tau_2}_{ij}:=a_{(i+\tau_1)(j+\tau_2)}$ where the additions are done on
$\mathbb{Z}_n$. The output of the convolutional layer gives basis functions that satisfy
\begin{equation}
\label{TI filter}
\tilde\phi_{kij} (a^{\tau_1 \tau_2}) = \tilde\phi_{k(i-\tau_1)(j-\tau_2)}(a) ,\quad k=1,\ldots,\alpha, \  i,j=1,\ldots,n, \ \forall \tau_1,\tau_2=1,\ldots,n.
\end{equation}
When using the architecture in Fig. \ref{figure: onelayer}, for any $\tau_1, \tau_2$,
\begin{eqnarray}
\label{trans symmetry}
f(a^{\tau_1 \tau_2}) &=& \sum_{k=1}^\alpha \beta _k \sum_{i=1}^n \sum_{j=1}^n \bigg( \tilde
\phi_{kij}(a^{\tau_1\tau_2})\bigg) = \sum_{k=1}^\alpha \beta _k \sum_{i=1}^n \sum_{j=1}^n \bigg( \tilde
\phi_{k(i-\tau_1)(j-\tau_2)}(a)\bigg) \cr
&=& \sum_{k=1}^\alpha \beta _k \phi_k(a),\quad
\phi_k:=\sum_{i=1}^n \sum_{j=1}^n \tilde \phi_{kij},
\end{eqnarray}
where $\beta_k$'s are the weights of the last densely connected layer. The summation over $i,j$
comes from the sum-pooling operation. Therefore, \eqref{trans symmetry} shows that the translational
symmetry of $f$ is preserved.

We note that all operations in Fig. \ref{figure: onelayer} are standard except the padding
operation. Typically, zero-padding is used to enlarge the size of the input in image classification
task, whereas we extend the input periodically due to the assumed periodic boundary condition.

\begin{figure}
	\centering
	\includegraphics[width=0.7\textwidth]{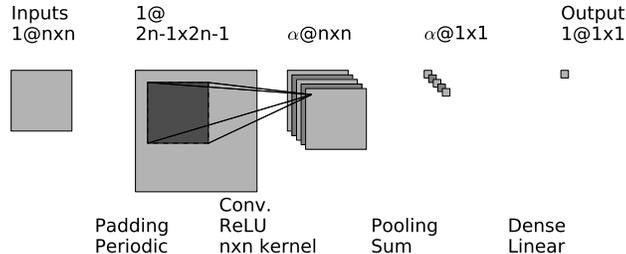}
	\caption{Single convolutional layer neural network for representing translational invariant
		function. }\label{figure: onelayer}
\end{figure}

\subsection{Implementation}\label{section: implementation}
The neural-network is implemented using Keras \cite{chollet2017keras}, an application programming
interface running on top of TensorFlow \cite{abadi2016tensorflow} (a library of toolboxes for
training neural-network). We use a mean-squared-error loss function. The optimization is done using
the NAdam optimizer \cite{dozat2016incorporating}. The hyper-parameter we tune is the learning rate,
which we lower if the training error fluctuates too much. The weights are initialized randomly from
the normal distribution. The input to the neural-network is whitened to have unit variance and
zero-mean on each dimension. The mini-batch size is always set to between 50 and 200.

\subsection{Numerical examples} \label{section: numerics}
\subsubsection{Effective conductance}
For the case of effective conductance, we assume the entries $a_i$'s of $a\in \mathbb{R}^{n^d}$ are
independently and identically distributed according to $\mathcal{U}[0.3, 3]$ where
$\mathcal{U}[\lambda_0,\lambda_1]$ denotes the uniform distribution on the interval
$[\lambda_0,\lambda_1]$. The results of learning the effective conductance function are presented in
Table \ref{table: homogenization}. To get the training samples, we solve the linear system in
\eqref{discrete elliptic}. We use the same number of samples for training and validation.  Both the
training and validation error are measured by
\begin{equation}
\sqrt{\frac{\sum_k (h_\theta(a^k) - A_{\mathrm{eff}}(a^k))^2}{\sum_k A_{\mathrm{eff}}(a^k)^2}},
\end{equation}
where $a^k$'s can either be the training or validation samples sampled from the same distribution
and $h_\theta$ is the neural-network-parameterized approximation function. We remark that although
incorporating domain knowledge in PDE to build a sophisticated neural-network architecture would
likely boost the approximation quality, such as what we do in the constructive proof for Theorem
\ref{theorem: effective conductance}, our results in Table \ref{table: homogenization} show that
even a simple network as in Fig. $\ref{figure: onelayer}$ can already give decent results with near
$10^{-3}$ accuracy. The simplicity of the NN is particularly important when using it as a surrogate
model of the PDE to generate samples cheaply.


\begin{table}[ht]
	\caption{Error in approximating the effective conductance function $A_{\mathrm{eff}}(a)$ in 2D. The
		mean and standard deviation of the effective conductance are computed from the samples in
		order to show the variability. The sample sizes for training and validation are the
		same.}\label{table: homogenization}
	\centering 
	\begin{tabular}{c c c c c c c c} 
		\hline\hline 
		$n$ & $\alpha$ &\begin{tabular}{@{}c@{}}Training \\ error\end{tabular} & \begin{tabular}{@{}c@{}}Validation \\ error\end{tabular}  & Average $A_{\mathrm{eff}}$ & \begin{tabular}{@{}c@{}}No. of \\ samples\end{tabular} & \begin{tabular}{@{}c@{}} No. of \\ parameters\end{tabular} \\ [.5ex] 
		\hline\hline 
		8 & 16 & $2.4\times 10^{-3}$ & $3.0\times 10^{-3}$ & $1.86\pm 0.10$ & $1.2\times 10^4$ & 1057 \\
		16 & 16 & $2.1\times 10^{-3}$ & $2.2\times 10^{-3}$ & $1.87\pm 0.052$ & $2.4\times 10^4$ &  4129 \\ [.5ex] 
		\hline 
	\end{tabular}
\end{table}

Before concluding this subsection, we use the exercise of determining the effective conductance in
1D to provide another motivation for the usage of a neural-network. Unlike the 2D case, in 1D the
effective conductance can be expressed analytically as the harmonic mean of $a_i$'s:
\begin{equation}
\label{harmonic mean}
A_{\mathrm{eff}}(a) = \biggl(\frac{1}{n}\sum_{i=1}^{n} \frac{1}{a_i} \biggr)^{-1}.
\end{equation}
This function indeed approximately corresponds to the deep neural-network shown in Fig. \ref{figure:
	homo1D}. The neural-network is separated into three stages. In the first stage, the approximation
to function $1/a_i$ is constructed for each $a_i$ by applying a few convolution layers with size $1$
kernel window. In this stage, the channel size for these convolution layers is chosen to be 16
except the last layer since the output of the first stage should be a vector of size $n$. In the
second stage, a layer of sum-pooling with size $n$ window is used to perform the summation in
(\ref{harmonic mean}), giving a scalar output. The third and first stages have the exact same
architecture except the input to the third stage is a scalar. $2560$ samples are used for training
and another $2560$ samples are used for validation. We let $a_i\sim \mathcal{U}[0.3, 1.5]$, giving
an effective conductance of $0.77\pm0.13$ for $n=8$. We obtain $4.9\times 10^{-4}$ validation error
with the neural-network in Fig. \ref{figure: homo1D} while with the network in Fig. \ref{figure:
	onelayer}, we get $5.5\times 10^{-3}$ accuracy with $\alpha=16$. As a check, in Fig.~\ref{figure:
	recip} we show that the output from the first stage is well-fitted by the reciprocal function.
\begin{figure}
	\centering
	\includegraphics[trim=0 80bp 0 60bp,clip,width=0.75\textwidth]{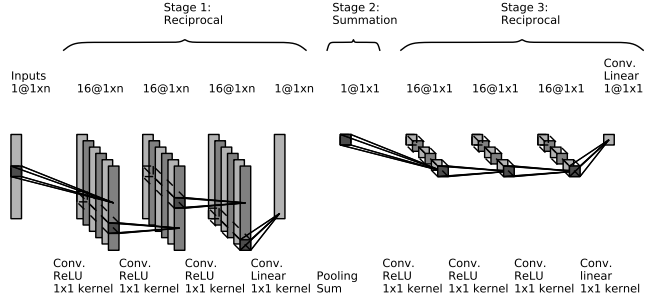}
	\caption{Neural-network architecture for approximating $A_{\mathrm{eff}}(a)$ in the 1D
		case. Although the layers in third stage are essentially densely-connected layers, we still
		identify them as convolution layers to reflect the symmetry between the first and third
		stages.}\label{figure: homo1D}
\end{figure}

\begin{figure}
	\centering
	\includegraphics[width=0.5\textwidth]{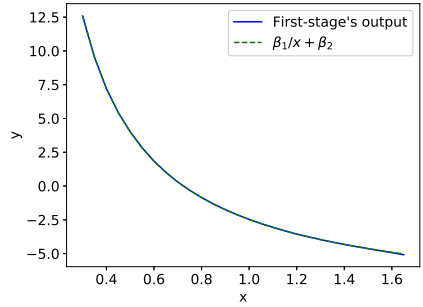}
	\caption{The first stage's output of the neural-network in Fig. \ref{figure: homo1D} fitted by
		$\beta_1/x + \beta_2$. The training result agrees well with the analytical structure of the solution to the 1D effective conductance.}\label{figure: recip}
\end{figure}

\subsubsection{Ground state energy of NLSE}
We next focus on the 2D NLSE example \eqref{discrete NLSE} with $\sigma=2$. The goal here is to
obtain a neural-network parametrization for $E_0(a)$, with input now being $a\in \mathbb{R}^{n^2}$
with i.i.d. entries distributed according to $\mathcal{U}[1, 16]$. As mentioned before, solving for
the ground state energy in NLSE is more expensive than solving for the effective conductance as in
this case, we need to solve a system of nonlinear equations. In order to generate training samples,
for each realization of $a$, the nonlinear eigenvalue problem \eqref{discrete NLSE} is solved via a
homotopy method. In such method, a sequence of NLSE $L u + a_i u_i + s u_i^3 = E_0 u_i\ \forall i$
with the normalization constraint on $u$ is solved with $s = s_1,\ldots, s_K$ where
$0=s_1<s_2<\ldots<s_K=\sigma$. First, the case $s=0$ is solved as a standard eigenvalue
problem. Then for each $s_i$ with $i>1$, Newton's method is used to solved the NLSE and $u,E_0$
obtained with $s = s_i$ will be used to warm start the Newton's iteration for $s_{i+1}$. In our
example, we change $s$ from 0 to 2 with a step size equals to $0.4$. The results are presented in
Table \ref{table: NLSE}.

\begin{table}[ht]
	\centering 
	\caption{Error in approximating the lowest energy level $E_0(a)$ for $n=8,16$ discretization.}\label{table: NLSE} 
	\begin{tabular}{c c c c c c c c} 
		\hline\hline 
		$n$ & $\alpha$ &\begin{tabular}{@{}c@{}}Training \\ error\end{tabular} & \begin{tabular}{@{}c@{}}Validation \\ error\end{tabular}  &  Average $E_0$  & \begin{tabular}{@{}c@{}}No. of \\ samples\end{tabular} & \begin{tabular}{@{}c@{}} No. of \\ parameters\end{tabular} \\ [0.5ex] %
		\hline\hline 
		8 & 5 &$4.9\times 10^{-4}$ & $5.0\times 10^{-4}$ & $10.48\pm 0.51$ & $4800 $ & 331\\
		16 & 5 & $1.5\times 10^{-4}$ & $1.5\times 10^{-4}$ & $10.46\pm 0.27$ & $1.05\times 10^4$ & 1291\\ [1ex] 
		\hline 
	\end{tabular}
\end{table}

\section{Conclusion}
\label{section: conclusion}

In this note, we present method based on deep neural-network to solve PDE with inhomogeneous
coefficient fields. Physical quantities of interest are learned as a function of the coefficient
field. Based on the time-evolution technique for solving PDE, we provide theoretical motivation to
represent these quantities using an NN. The numerical experiments on elliptic equation and NLSE show
the effectiveness of simple convolutional neural network in parameterizing such function to
$10^{-3}$ accuracy. We remark that while many questions should be asked, such as what is the best
network architecture and what situations can this approach handle, the goal of this short note is
simply to suggest neural-network as a promising tool for model reduction when solving PDEs with
uncertainties.

\bibliographystyle{plain}
\bibliography{bibref}

\newpage

\appendix
\section*{Supplementary material: Proof of representability of effective conductance by NN}
As mentioned previously in Section \ref{section: theory}, the first step of constructing an NN to
represent the effective conductance is to perform time-evolution iterations in the form of
\eqref{discrete time heat equation}. However, since at each step we need to approximate the map from
$u^m$ to $u^{m+1}$ in \eqref{discrete time heat equation} using NN, the process of time-evolution is
similar to applying noisy gradient descent on $\mc{E}(u;a)$. More precisely, after performing a step
of gradient descent update, the NN approximation incur noise to the update, i.e.
\begin{eqnarray}
\label{noisy heat equation}
v^0 = u^0 =0,\quad u^{m+1} = v^m - \Delta t \nabla \mc{E}(v^m),\quad v^{m+1} = u^{m+1} + \Delta t
\varepsilon^{m+1}.
\end{eqnarray}
Here $\mc{E}(u;a)$ is abbreviated as $\mc{E}(u)$, and $\varepsilon^{m+1}$ is the error for each
layer of the NN in approximating each exact time-evolution iteration $u^{m+1} $. To be sure, instead
of $u^m$, the object that is evolving in the NN as $m$ changes is $v^m$.

\begin{assumption}
	\label{assumption: spectral bound}
	We assume $a \in \mathcal{A} = \{ a \in\mathbb{R}^{n^d} \mid a_i \in [\lambda_0,
	\lambda_1],\,\forall i\}$ with $\lambda_0 >0$. Under this assumption $\lambda_a := \| L_a \|_2$
	and $\mu_a:= \frac{1}{\| L_a^\dagger \|_2}$ satisfy
	\begin{equation} 
	\label{assumptions}
	\lambda_a = O(\lambda_1 h^{d-2}),\quad \mu_a = \Omega(\lambda_0 h^d).
	\end{equation}
	Here for matrices $\|\cdot \|_2$ denotes the spectral norm, $h=1/n$.
\end{assumption}

\begin{assumption}
	We assume the NN results an approximation error term $\varepsilon^{m+1}$ with properties
	\begin{equation}
	\label{assumptions 2}
	\| \varepsilon^{m+1} \|_2\leq c\|\nabla \mc{E}(v^m) \|_2, \quad {\mathbf{1}}^\top
	\varepsilon^{m+1}=0,\quad m = 0,\ldots M-1,
	\end{equation}
	when approximating each step of time-evolution.
\end{assumption}

\begin{lemma}
	\label{lemma: sufficient descent}
	The iterations in \eqref{noisy heat equation} satisfies
	\begin{equation}
	\mc{E}(v^{m+1}) - \mc{E}(v^m) \leq - \frac{\Delta t}{2}\|\nabla \mc{E}(v^m) \|_2^2,
	\end{equation}
	if $\Delta t \leq \delta, \delta = \big(1 - \frac{1}{2(1-c)}\big)\frac{2}{\lambda_a'}$ with
	$\lambda_a' = (1+\frac{c^2}{1-c})\lambda_a$. Furthermore,
	\begin{eqnarray}
	\label{grad bound}
	\frac{\Delta t}{2} \sum_{m=0}^{M-1} \| \nabla \mc{E}(v^{m+1}) \|_2^2 \leq  \mc{E}(v^0)-\mc{E}(v^M)\leq \mc{E}(v^0)-\mc{E}(u^*).
	\end{eqnarray}
\end{lemma}
\begin{proof}
	From Lipshitz property of $\nabla \mc{E}(u)$ \eqref{assumptions},
	\begin{eqnarray}
	\mc{E}(v^{m+1}) - \mc{E}(v^m) &\leq& \langle \nabla \mc{E}(v^m), v^{m+1}-v^m \rangle + \frac{\lambda_a}{2} \|  v^{m+1} - v^m \|_2^2	\cr
	& = & \langle  \nabla \mc{E}(v^m), v^m - \Delta t (\nabla \mc{E}(v^m) + \varepsilon^{m+1})-v^m \rangle \cr 
	&\ & +  \frac{\lambda_a}{2} \|v^m - \Delta t (\nabla \mc{E}(v^m) + \varepsilon^{m+1})-v^m\|_2^2 \cr
	& = & -\Delta t (1 - \frac{\Delta t \lambda_a}{2}) \|\nabla \mc{E}(v^m)\|_2^2 \cr 
	&\ &+ \Delta t(1 - \frac{\Delta t \lambda_a }{2})\langle \varepsilon^{m+1}, \nabla \mc{E}(v^m)\rangle + \frac{\lambda_a \Delta t^2}{2} \| \varepsilon^m \|_2^2\cr
	&\leq& -\Delta t (1 - \frac{\Delta t \lambda_a}{2}) \|\nabla \mc{E}(v^m)\|_2^2  \cr 
	&\ &+ c\Delta t (1 - \frac{\Delta t \lambda_a }{2} +  \frac{c \Delta t \lambda_a }{2}) \|\nabla \mc{E}(v^m)\|_2^2\cr
	&=& -\Delta t \big((1 - c) - (1-c+c^2) \frac{\Delta t \lambda_a}{2} \big) \|\nabla \mc{E}(v^m)\|_2^2\cr
	&=& -\Delta t (1-c) \big(1 - \frac{1-c+c^2}{1-c}\frac{\Delta t \lambda_a}{2} \big) \|\nabla \mc{E}(v^m)\|_2^2\cr
	&=& -\Delta t (1-c) \big(1 - \frac{\Delta t \lambda_a'}{2} \big) \|\nabla \mc{E}(v^m)\|_2^2.
	\end{eqnarray}
	Letting $\Delta t \leq \big(1 - \frac{1}{2(1-c)}\big)\frac{2}{\lambda_a'}$, we get
	\begin{equation}
	\mc{E}(v^{m+1}) - \mc{E}(v^m) \leq  -\frac{\Delta t}{2} \|\nabla \mc{E}(v^m)\|_2^2.
	\end{equation}
	Summing the LHS and RHS gives \eqref{grad bound}. This concludes the lemma.
\end{proof}

\begin{theorem}
	If $\Delta t$ satisfies the condition in Lemma \ref{lemma: sufficient descent}, given any
	$\epsilon>0$, $\vert \mc{E}(v^M) - \mc{E}(v) \vert \leq \epsilon$ for $M =
	O((\frac{\lambda_1^2}{\lambda_0}+\lambda_1)\frac{n^2}{\epsilon})$.
\end{theorem}
\begin{proof}
	Since by convexity
	\begin{eqnarray}
	\mc{E}(u^*) - \mc{E}(v^m)&\geq &\langle \nabla \mc{E}(v^m), u^* -  v^m \rangle, 
	\end{eqnarray}
	along with Lemma \ref{lemma: sufficient descent}, 
	\begin{eqnarray}
	\mc{E}(v^{m+1}) &\leq& \mc{E}(u^*) + \langle \nabla \mc{E}(v^m),  v^m - u^* \rangle  - \frac{\Delta t}{2}  \|\nabla \mc{E}(v^m)\|_2^2\cr
	&=& \mc{E}(u^*) + \frac{1}{2\Delta t}\big(2\Delta t\langle \nabla \mc{E}(v^m),  v^m - u^* \rangle  - \Delta t^2  \|\nabla \mc{E}(v^m)\|_2^2 \cr 
	&\ &+  \| v^m - u^* \|_2^2 -  \| v^m - u^* \|_2^2\big) \cr
	&=& \mc{E}(u^*) + \frac{1}{2\Delta t} ( \| v^m - u^* \|_2^2 - \| v^m - \Delta t \nabla \mc{E}(v^m)-u^*\|_2^2)\cr
	&=&   \mc{E}(u^*) +  \frac{1}{2\Delta t} ( \| v^m - u^* \|_2^2   - \|v^{m+1} - \Delta t \varepsilon^{m+1}-u^*\|_2^2)\cr
	&=&  \mc{E}(u^*) +  \frac{1}{2\Delta t} ( \| v^m - u^* \|_2^2 - \| v^{m+1}- u^* \|_2^2 \cr 
	&\ &+ 2\Delta t \langle \varepsilon^{m+1}, v^{m+1}-u^*\rangle -\Delta t^2 \|\varepsilon^{m+1}\|_2^2)\cr
	&=& \mc{E}(u^*) +  \frac{1}{2\Delta t} \big( \| v^m - u^* \|_2^2 - \| v^{m+1}- u^* \|_2^2 + \Delta t^2\| \varepsilon^{m+1}\|_2^2 \cr 
	&\ &+ 2\Delta t \langle \varepsilon^{m+1}, v^m-u^*\rangle - 2\Delta t \langle \varepsilon^{m+1}, \nabla \mc{E}(v^m)\rangle\big)\cr
	&\leq& \mc{E}(u^*) +  \frac{1}{2\Delta t} \big( \| v^m - u^* \|_2^2 - \| v^{m+1}- u^* \|_2^2 + \Delta t^2 \| \varepsilon^{m+1}\|_2^2  \cr 
	&\ &+  2\Delta t \| \varepsilon^{m+1}\|_2 (\|v^m-u^*\|_2 +\|\nabla \mc{E}(v^m)\|_2)\big)\cr
	&\leq& \mc{E}(u^*) + \frac{1}{2\Delta t} \big( \| v^m - u^* \|_2^2 - \| v^{m+1}- u^* \|_2^2 +  \Delta t^2 \| \varepsilon^{m+1}\|_2^2 \cr 
	&\ &+ 2\Delta t (1+ 2/\mu_a )\| \varepsilon^{m+1}\|_2 \| \nabla \mc{E}(v^m) \|_2\big)\cr
	&\leq& \mc{E}(u^*) + \frac{1}{2\Delta t} \big( \| v^m - u^* \|_2^2 - \| v^{m+1}- u^* \|_2^2 +  c^2 \Delta t^2  \| \nabla \mc{E}(v^m) \|_2^2 \cr 
	&\ &+ 2c (1+ 2/\mu_a ) \Delta t \| \nabla \mc{E}(v^m) \|_2^2\big).\label{bound 1}
	\end{eqnarray}
	The last second inequality follows from \eqref{assumptions}, which implies $ \| L_a u\|_2 \geq
	\mu_a \|u\|_2$ if $u^\top \mathbf{1} = 0$. More precisely, the fact that $v^0 =0, \nabla
	\mc{E}(u)^\top \mathbf{1} = 0$ (follows from the form of $L_a$ and $b_a$ defined in
	\eqref{discrete elliptic}), and ${\varepsilon^m}^\top \mathbf{1} = 0\ \forall m$ (due to the
	assumption in \eqref{assumptions 2}) implies ${v^m}^\top \mathbf{1} = 0$, hence $\frac{\mu_a}{2}
	\| v^m-u^*\|_2 \leq \| \nabla \mc{E}(v^m) - \nabla \mc{E}(u^*) \|_2 = \| \nabla \mc{E}(v^m)
	\|_2$.  Reorganizing \eqref{bound 1} we get
	\begin{multline}
	\mc{E}(v^{m+1}) - \mc{E}(u^*) \cr
	\leq \frac{1}{2\Delta t} \bigg( \| v^m - u^* \|_2^2 - \| v^{m+1}- u^* \|_2^2 +  c\Delta t\big(c\Delta t + 2(1+\frac{2}{\mu_a})\big)\|\nabla \mc{E}(v^m)\|_2^2\bigg).
	\end{multline}
	Summing both left and right hand sides results in
	\begin{multline}
	\mc{E}(v^{M}) - \mc{E}(u^*) \leq \frac{1}{M} \sum_{m=0}^{M-1} \mc{E}(v^{m+1}) - \mc{E}(u^*)\cr 
	\leq	\frac{1}{M}\bigg[\bigg(\frac{\| v^0 - u^*\|^2_2}{2\Delta t} \bigg)+ \frac{2c}{\Delta t}\bigg(c\Delta t + 2(1+\frac{2}{\mu_a})\bigg)(\mc{E}(v^0)-\mc{E}(u^*))\bigg]
	\end{multline}
	where the second inequality follows from \eqref{grad bound}.  In order to derive a bound for
	$\|v^0 - u^*\|_2^2$, we appeal to strong convexity property of $\mc{E}(u)$:
	\begin{equation}
	\mc{E}(v^0) - \mc{E}(u^*) \geq \langle \nabla \mc{E}(u^*),v^0-u^*\rangle+\frac{\mu_a}{2} \| v^0-u^*\|_2^2 = \frac{\mu_a}{2} \| v^0-u^*\|_2^2 
	\end{equation}
	for $\langle \mathbf{1}, v^0-u^*\rangle=0$. The last equality follows from the optimality of $u^*$. Then 
	\begin{eqnarray}
	\mc{E}(v^{M}) - \mc{E}(u^*)  \leq	\frac{1}{M}\bigg[\bigg(\frac{1}{\mu_a\Delta t} \bigg)+ \frac{2c}{\Delta t}\bigg(c\Delta t + 2(1+\frac{2}{\mu_a})\bigg)\bigg](\mc{E}(v^0)-\mc{E}(u^*)).
	\end{eqnarray}
	Since $\mc{E}(v^0) = h^d\frac{a^\top \mathbf{1}}{2}=O(\lambda_1)$, along with
	$\lambda_a=O(\lambda_1 h^{d-2})$ and $\mu_a=\Omega(\lambda_0 h^d)$ (Assumption
	\ref{assumption: spectral bound}), we establish the claim.
\end{proof}

\end{document}